\documentclass[12pt]{amsart}
\usepackage[osf,sc]{mathpazo}
\usepackage{amssymb}

\usepackage{geometry}
\usepackage{bm}
\usepackage{graphicx}
\usepackage{tikz-cd}
\usepackage{hyperref}
\hypersetup{
    colorlinks=true, 
    linktoc=all,     
    linkcolor=blue,
        citecolor=red,
    filecolor=black,
    urlcolor=blue	
}
\usepackage{enumerate}
\usepackage[inline]{enumitem}
\makeatletter
\newcommand{\inlineitem}[1][]{%
\ifnum\enit@type=\tw@
    {\descriptionlabel{#1}}
  \hspace{\labelsep}%
\else
  \ifnum\enit@type=\z@
       \refstepcounter{\@listctr}\fi
    \quad\@itemlabel\hspace{\labelsep}%
\fi} \makeatother
\newcommand{\ga}{\alpha}
\newcommand{\gb}{\beta}
\newcommand{\gga}{\gamma}

\newcommand{\gp}{\pi}

\newcommand{\Gd}{\Delta}

\newcommand{\Gl}{\Lambda}

\newcommand{\Gs}{\Sigma}

\newcommand{\Gom}{\Omega}

\newcommand{\subs}{\subset}
\newcommand{\sups}{\supset}

\newcommand{\bs}{\backslash}

\newcommand{\mbb}{\mathbb}

\newcommand{\us}{\underset}
\newcommand{\os}{\overset}

\newcommand{\lra}{\longrightarrow}
\newcommand{\llra}{\longleftrightarrow}

\newcommand{\Ra}{\Rightarrow}

\newcommand{\es}{\emptyset}

\newcommand{\equ}[1]{%
\begin{equation*}
#1
\end{equation*}
}
\newcommand{\equa}[1]{%
\begin{equation*}
\begin{aligned}
#1
\end{aligned}
\end{equation*}
}

\newcommand{\matthreefour}[9]{%
  \def\argi{{#1}}%
  \def\argii{{#2}}%
  \def\argiii{{#3}}%
  \def\argiv{{#4}}%
  \def\argv{{#5}}%
  \def\argvi{{#6}}%
  \def\argvii{{#7}}%
  \def\argviii{{#8}}%
  \def\argix{{#9}}%
  \matthreefourRelay
}
\newcommand\matthreefourRelay[3]{%
\begin{pmatrix}
  \argi     & \argii & \argiii & \argiv\\
  \argv     & \argvi & \argvii & \argviii\\
  \argix    & #1     & #2      & #3
\end{pmatrix}
}

\theoremstyle{plain}
\newtheorem{theorem}{Theorem}[section]

\newtheorem{prop}[theorem]{Proposition}
\newtheorem{lem}[theorem]{Lemma}
\newtheorem{cor}[theorem]{Corollary}
\newtheorem{conj}[theorem]{Conjecture}

\newtheorem{claim}[theorem]{Claim}
\newtheorem{prob}[theorem]{Problem}

\makeatletter
\def\namedlabel#1#2{\begingroup
   \def\@currentlabel{#2}%
   \label{#1}\endgroup
}
\makeatother

\newtheorem*{thmOmega}{\bf{Theorem} $\bm{\Gom}$}

\theoremstyle{definition}
\newtheorem{defn}[theorem]{Definition}

\theoremstyle{remark}
\newtheorem{remark}[theorem]{Remark}
\newtheorem{example}[theorem]{Example}
\numberwithin{equation}{section}

\begin{document}
\title[On a Conjecture of Kelly on Sylvester Gallai Designs]{On a Conjecture of Kelly on (1,3)-representation of Sylvester Gallai Designs}
\author{C P Anil Kumar}
\address{C P Anil Kumar, School of Mathematics, Harish-Chandra Research Institute, HBNI, Chhatnag Road, Jhunsi, Prayagraj (Allahabad), 211 019,  India. \,\, email: {\tt akcp1728@gmail.com}}
\author{Anoop Singh}

\address{Anoop Singh, School of Mathematics, Harish-Chandra Research Institute, HBNI, Chhatnag Road, Jhunsi, Prayagraj (Allahabad), 211 019,  India. \,\, email: {\tt anoopsingh@hri.res.in}}
\thanks{This work is done while the first author is a Post Doctoral Fellow and the second author is a Research Scholar at Harish-Chandra Research Institute, Prayagraj(Allahabad).}
\date{\sc \today}
\begin{abstract}
We give an exact criterion of a conjecture of L.~M.~Kelly to hold true which is stated as follows. If there is a finite family $\Gs$ of mutually skew lines in $\mbb{R}^l,l\geq 4$ such that the three dimensional affine span (hull) of every two lines in $\Gs$, contains at least one more line of $\Gs$, then we have that $\Gs$ is entirely contained in a three dimensional space if and only if the arrangement of affine hulls is central.
Finally, this article leads to an analogous question for higher dimensional skew affine spaces, that is, for $(2,5)$ representations of sylvester-gallai designs in $\mbb{R}^6$, which is answered in the last section.
\end{abstract}
\maketitle
\section{\bf{Historical Introduction and the Main Theorem}}
In 1893, J.~J.~Sylvester~\cite{JJS} posed the following problem that has attracted considerable attention.
\begin{prob}
Let $n$ given points have the property that the line joining any two of them passes through a third point of the set. Must the $n$ points lie on a line?
\end{prob}
For a few decades this problem was forgotten and forty years later, in 1933, Erd\H{o}s asked this question~\cite{MR0698647}, which was solved by a fellow countryman T.~Gallai and others~\cite{MR1525919}, (also see~\cite{MR1069788} for a survey).

Let $\mbb{K}$ be a field and $\mbb{P}^n(\mbb{K})$ be the $n$-dimensional projective space. We say a finite set $A \subseteq \mbb{P}^n(\mbb{K})$ of points is a Sylvester-Gallai configuration (SGC), if for any distinct points $x,y\in A$ there exists $z\in A$ such that $x,y,z$ are distinct collinear points. 
An easy generalization of the Sylvester's problem gives the following theorem.
\begin{theorem}[Sylvester-Gallai Theorem]
If $A$ is a finite noncollinear set of points in the $n$-dimensional projective space $\mbb{P}^n(\mbb{R})$ then there exists a line through exactly two points of $A$, that is, every SGC in $\mbb{P}^n(\mbb{R})$ is collinear.
\end{theorem}
It is a classical fact that the nine inflection points of a non-degenerate cubic curve in $\mbb{P}^2(\mbb{C})$ constitute an SGC. Serre~\cite{JPSERRE} asked whether an SGC in $\mbb{P}^n(\mbb{C})$ must be coplanar, that is, it must lie in a two dimensional complex plane. This has been solved by L.~M.~Kelly~\cite{MR0834051} using an inequality of F.~Hirzebruch~\cite{MR0717609} involving the number of incidences of points and lines in $\mbb{P}^2(\mbb{C})$. Y.~Miyaoka~\cite{MR0460343} and S.~T.~Yau~\cite{MR0451180} have proved that, for algebraic surfaces $S$ of general type which are minimal, the inequality $c_1^2(S)\leq 3c_2(S)$, where $c_1^2(S),c_2(S)$ are the chern numbers of $S$. Using the Miyaoka-Yau inequality, F.~Hirzebruch~\cite{MR0717609} has deduced that, for an arrangement of $k$ lines in the complex projective plane with $t_k=t_{k-1}=0$, the inequality $t_2+t_3\geq k+t_5+2t_6+3t_7+\cdots$, where $t_r$ is the number of points where exactly $r$ lines meet. Using Hirzebruch inequality and the fact that, for a field $\mbb{K}$ of characteristic zero, any non-linear SGC in $\mbb{P}^2(\mbb{K})$ cannot be contained in three concurrent lines intersecting at a point of SGC, the following theorem is deduced. 	
\begin{theorem}[Kelly]
\label{theorem:LMKELLY}
Every SGC in $\mbb{P}^n(\mbb{C})$ is coplanar.
\end{theorem}
Also generalizing an SGC, more liberal representations of them have been studied where  points are represented by affine or projective subspaces of dimension $m$ and lines are represented by affine or projective subspaces of dimension $n$ in a higher dimensional affine or projective space respectively. Such a representation is called an $(m,n)$-representation.

In the hope to obtain an elementary solution to Question~\cite{JPSERRE} of Serre, L.~M.~Kelly in 1986 has mentioned the following conjecture in~\cite{MR0834051} regarding a $(1,3)$-represent-ation in four dimensional euclidean space.
 
\begin{conj}
\label{conj:LMKELLYONE}
If there is a finite family $\Gs$ of mutually skew lines in $\mbb{R}^4$ such that the three dimensional affine span (hull) of every two lines in $\Gs$ , contains at least one more line of $\Gs$, then $\Gs$ is entirely contained in a three dimensional space.
\end{conj} 

Three years later in 1989, the authors E.~Boros, Z.~F\"{u}redi and L.~M.~Kelly~\cite{MR0996767} have disproved  Conjecture~\ref{conj:LMKELLYONE} by actually giving a $(1,3)$-representation of the seven point projective plane $\mbb{P}^2(\mbb{F}_2)$ in $\mbb{R}^4$ where points of $\mbb{P}^2(\mbb{F}_2)$ are represented by mutually skew lines in $\mbb{R}^4$ and lines of $\mbb{P}^2(\mbb{F}_2)$ are represented by hyperplanes in $\mbb{R}^4$. We will mention in Example~\ref{example:FanoPlane} of this article, a plenty of such $(1,3)$-representations.

In a further effort to obtain an elementary proof of Serre's Question~\cite{JPSERRE}, the authors in~\cite{MR0996767} have mentioned the following conjecture.

\begin{conj}
\label{conj:LMKELLYTWO}
If a finite family of pairwise skew lines in $\mbb{R}^4$ is such that the three dimensional affine span (hull) containing any two of the lines contains at least two more, then the family is in a single three dimensional space.
\end{conj}

However again in 1992,  J.~Bokowski and J.~R.~Gebert~\cite{MR1142273} have constructed a $(1,3)$-representation of the thirteen point projective plane $\mbb{P}^2(\mbb{F}_3)$ in $\mbb{R}^4$ negating Conjecture~\ref{conj:LMKELLYTWO}. Then it is believed for a while that the answer to Serre's Question~\cite{JPSERRE} perhaps does require deeper algebraic properties of complex numbers. More than a decade later in 2006, N.~Elkies, L.~M.~Pretorius, K.~J.~Swanpoel~\cite{MR2202107} have given an elementary proof of Theorem~\ref{theorem:LMKELLY}. They in fact have proved much more and thereby have concluded the quest for an elementary proof.

\subsection{\bf{The statement of the Main Theorem}}
Conjecture~\ref{conj:LMKELLYONE} in~\cite{MR0834051} has been forgotten after it has been corrected by L.~M.~Kelly in~\cite{MR0996767}. This article now settles in a better manner enhancing the importance of this conjecture which has been made by L.~M.~Kelly in~\cite{MR0834051}. The main theorem concerning Conjecture~\ref{conj:LMKELLYONE} is stated as follows.

\begin{thmOmega}
\namedlabel{theorem:Main}{$\Gom$}
Let $l\geq 4$ be an integer.
Suppose there is a finite family $\Gs$ of mutually skew lines in $\mbb{R}^l$ such that the three dimensional affine span (hull) of every two lines in $\Gs$, contains at least one more line of $\Gs$. Then we have that $\Gs$ is entirely contained in a three dimensional space if and only if the arrangement of the affine hulls is central, that is, the intersection of all the three dimensional affine hulls in $\mbb{R}^l$ of every two lines in $\Gs$ is non-empty.

Stating in other words as follows: If there is a finite family $\Gs$ of mutually skew lines in $\mbb{R}^l$ such that the three dimensional affine span (hull) of every two lines in $\Gs$ is a linear subspace (containing origin) and contains at least one more line of $\Gs$ then we have that $\Gs$ is entirely contained in a three dimensional linear subspace in $\mbb{R}^l$.
\end{thmOmega}

More about recent incidence and arrangement problems in particular of Sylvester-Gallai-Type can be found in P.~Brass, W.~O.~J.~Moser and P.~Janos~\cite{MR2163782} along with a huge list of references at the end section $7.2$ on pages 307-310.

\section{\bf{On small $(1,3)$-representations in $\mbb{R}^4$}}
We start this section with the following lemma which is an easy observation.

\begin{lem}
For $1\leq n\leq 6$, if any $n$ mutually skew lines in $\mbb{R}^4$ satisfies that the three dimensional affine span (hull) of every two lines contains a third, then all the $n$ lines are contained in a three dimensional affine space in $\mbb{R}^4$.
\end{lem}
\begin{proof}
The proof is immediate.	
\end{proof}
For $n=7$, the only $(1,3)$-representation scenario in $\mbb{R}^4$ which requires further examination is the $(1,3)$-representation of the seven-point projective plane $\mbb{P}^2(\mbb{F}_2)$. We discuss this example below.
\begin{example}
\label{example:FanoPlane}

Let $L_i,1\leq i\leq 7$ be mutually skew lines in $\mbb{R}^4$ which is a $(1,3)$-representation of the Fano plane. Let $H_1\sups L_1\cup L_2\cup L_3,\ H_2 \sups L_1\cup L_4\cup L_5,\ H_3\sups L_1\cup L_6\cup L_7,\ H_4\sups L_2\cup L_4 \cup L_6,\ H_5\sups L_2\cup L_5\cup L_7,\ H_6 \sups L_3\cup L_4\cup L_7,\ H_7\sups L_3\cup L_5\cup L_6$ where $H_i,1\leq i\leq 7$ are hyperplanes in $\mbb{R}^4$. Suppose the hyperplanes $H_i,1\leq i\leq 7$ are all distinct and form a generic hyperplane arrangement in $\mbb{R}^4$. Then we can recover the lines from the hyperplanes as follows. We have $L_1=H_1\cap H_2\cap H_3, L_2=H_1\cap H_4\cap H_5, L_3= H_1\cap H_6\cap H_7, L_4= H_2\cap H_4 \cap H_6, L_5= H_2\cap H_5\cap H_7, L_6=H_3\cap H_4\cap H_7, L_7= H_3\cap H_5\cap H_6$. We begin with a generic hyperplane arrangement $H_i,1\leq i\leq 7$ and define $L_i$ using the hyperplanes. It is clear that the lines $L_i, 1\leq i\leq 7$ do not intersect since the intersection of any five distinct hyperplanes is empty. The coplanarity conditions on any pair of lines $L_i,L_j,i\neq j$ is a polynomial condition on the coefficients defining the hyperplanes. So, if we assume the coefficients defining the hyperplanes are more generic enough we obtain a plenty of $(1,3)$-representations of the Fano plane. This proves that main Theorem~\ref{theorem:Main} does not hold true, if the hyperplane arrangement is not central.  
\end{example}
Now we prove the following proposition which is the first non-trivial case of main Theorem~\ref{theorem:Main}.
\begin{prop}
A $(1,3)$-representation of the Fano plane where the affine hulls of any two mutually skew lines are all linear subspaces containing origin in $\mbb{R}^4$ does not exist.
\end{prop}
\begin{proof}
Let $L_i,1\leq i\leq 7$ be mutually skew lines in $\mbb{R}^4$ which is a $(1,3)$-representa-tion of the Fano plane. Let $H_1\sups L_1\cup L_2\cup L_3, H_2 \sups L_1\cup L_4\cup L_5, H_3\sups L_1\cup L_6\cup L_7, H_4\sups L_2\cup L_4 \cup L_6, H_5\sups L_2\cup L_5\cup L_7, H_6 \sups L_3\cup L_4\cup L_7, H_7\sups L_3\cup L_5\cup L_6$ where $H_i,1\leq i\leq 7$ are hyperplanes in $\mbb{R}^4$. These are the affine hulls of the skew lines. If $H_i=H_j$ for some $1\leq i\neq j\leq 7$ then we have $H_1=H_2=H_3=H_4=H_5=H_6=H_7$. So assume that all the $H_i$ are distinct and hence for $1\leq i\neq j\leq 7, \dim_{\mbb{R}}(H_i\cap H_j)=2$. We also have that 
$0\in \us{i=1}{\os{7}{\cap}}H_i$. Assume, without loss of generality, that $L_1,L_2,L_3,L_4,L_6,L_7$ do not pass through origin and $L_5$ may or may not pass through origin. Then we have $2=\dim_{\mbb{R}}(H_1\cap H_2\cap H_3)=\dim_{\mbb{R}}(H_1\cap H_4\cap H_5)=\dim_{\mbb{R}}(H_1\cap H_6\cap H_7)=\dim_{\mbb{R}}(H_2\cap H_4 \cap H_6)=\dim_{\mbb{R}}(H_3\cap H_5\cap H_6)=\dim_{\mbb{R}}(H_3\cap H_4\cap H_7)=2$ and $\dim_{\mbb{R}}(H_2\cap H_5\cap H_7)\leq 2$. Let $v_i\in \mbb{R}^4$ be a non-zero normal vector of $H_i,1\leq i\leq 7$. Then each of the following six sets $\{v_1,v_2,v_3\},\{v_1,v_4,v_5\},\{v_1,v_6,v_7\},\{v_2,v_4,v_6\},\{v_3,v_5,v_6\},\{v_3,v_4,v_7\}$ span a two dimensional space in $\mbb{R}^4$. If $v_1,v_3,v_6$ are linearly dependent then $H_1\cap H_3\cap H_6=H_1\cap H_3=H_1\cap H_6=H_3\cap H_6$ which implies that $L_1,L_3,L_7$ are coplanar which is a contradiction. So $v_1,v_3,v_6$ are linearly independent. If $v_1,v_3,v_6,v_4$ are linearly independent then by applying a suitable linear transformation in $GL_4(\mbb{R})$ we can assume that $v_1=(1,0,0,0),v_3=(0,1,0,0),v_6=(0,0,1,0), v_4=(0,0,0,1)$. Then we have $v_2=(a,b,0,0)$
with $ab\neq 0$ and similarly $v_2=(0,0,x,y)$ with $xy\neq 0$ which is contradiction.
So $v_4$ is linearly dependent on the linear independent set $\{v_1,v_3,v_6\}$.
Let us assume again that $v_1=(1,0,0,0),v_3=(0,1,0,0),v_6=(0,0,1,0), v_4=(\ga,\gb,\gga,0)$. We observe that the product $\ga\gb\gga\neq 0$. For example, if $\gga=0$ then $H_1\cap H_3 \cap H_4=H_1\cap H_3 =H_1\cap H_4=H_3\cap H_4$. This implies that $L_1,L_2,L_6$ are coplanar which is a contradiction. Similarly $\ga\neq 0\neq \gb$.
So we get now that $v_2=(c,d,0,0)=cv_1+dv_3=pv_6+qv_4=(q\ga,q\gb,q\gga+p,0)$ for some $c,d,p,q\in \mbb{R}^{*}=\mbb{R}\bs\{0\}$. So $p=-q\gga$ and $v_2=(q\ga,q\gb,0,0)$. Similarly we have $v_7=(r\ga,0,r\gga,0),v_5=(0,s\gb,s\gga,0)$ for some $r,s\in \mbb{R}^{*}$. Now the following matrix
\equ{\matthreefour {q\ga}{q\gb}{0}{0}{r\ga}{0}{r\gga}{0}{0}{s\gb}{s\gga}{0}}
has rank three because $qrs\ga\gb\gga\neq 0$. This implies that $H_2\cap H_5\cap H_7=L_5$ is one dimensional and $L_5$ must pass through origin.
Also note that the line $L=\{t(0,0,0,1)\mid t\in \mbb{R}\}$ is perpendicular in $\mbb{R}^4$ to all $v_i,1\leq i\leq 7$. Hence we have $L\subs \us{i=1}{\os{7}{\cap}}H_i$. This implies that 
$L=L_5$. Now it follows that $L_5$ is coplanar with all $L_i,1\leq i\leq 7$ which is a contradiction. Hence the proposition follows. 
\end{proof}
Now as a consequence, we have the following corollary when $\mid\Gs\mid=7$ in Theorem~\ref{theorem:Main} for $l=4$.
\begin{cor}
	\label{cor:Fano}
Suppose there is a finite family $\Gs$ of seven mutually skew lines in $\mbb{R}^4$ such that the three dimensional affine span (hull) of every two lines in $\Gs$, contains at least one more line of $\Gs$. Then we have that $\Gs$ is entirely contained in a three dimensional space if and only if the hyperplane arrangement of affine hulls is central.
\end{cor}
\section{\bf{Proof of the Main Theorem for $n=4$}}
In this section, we prove the main theorem when $n=4$. First we prove the following. 

\begin{lem}
\label{lemma:AffineHulls}
Let $\Gs$ be a finite family of mutually skew lines in $\mbb{R}^n$ such that the three dimensional affine span (hull) of every two lines in $\Gs$, contains at least one more line of $\Gs$. Then either all the lines are contained in one three dimensional space or every line is contained in at least three different three dimensional affine hulls.
\end{lem}
\begin{proof}
Let $L$ be a line which is contained in exactly two three dimensional affine hulls $H_1$ and $H_2$. Then we have a partition of the set $\Gs\bs \{L\}=\Gs_1\cup \Gs_2$ where the line $L$ and all the lines in $\Gs_1$ are contained in three dimensional affine hull $H_1$ and the same line $L$ and all the lines in $\Gs_2$ are contained in another three dimensional hull $H_2$. Now we have $\mid \Gs_i\mid \geq 3,i=1,2$. Let $L_i\in \Gs_i,i=1,2$. Then the affine hull $\Gd $ of $L_1$ and $L_2$ must contain a third line from $\Gs$. So either the set $\Gd \cap (\Gs_1\cup\{L\})$ or the set $\Gd \cap (\Gs_2\cap\{L\})$ must have cardinality at least two which is a contradiction. Hence the lemma follows.     	
\end{proof}

Now we state a very important theorem which is required to prove main Theorem~\ref{theorem:Main} for $n=4$.
\begin{theorem}
\label{theorem:CoplanarLine}
Suppose there is a finite family $\Gs=\{L_1,L_2,\cdots,L_n\}$ of mutually skew lines in $\mbb{R}^4$ such that the three dimensional affine span (hull) of every two lines in $\Gs$, contains at least one more line of $\Gs$. Let $\{H_1,H_2,\cdots,H_m\}$ be the distinct three dimensional affine hulls and $m\geq 2$. If the hyperplane arrangement of the affine hulls is central, that is, $O\in \us{i=1}{\os{m}{\cap}} H_i$, then there exists a line $L\subs \mbb{R}^4$ different from  $L_i$ and passing through $O$ such that the line $L_i$ and $L$ are coplanar for each $1\leq i\leq n$.
\end{theorem}
We first prove the main theorem for $l=4$ using this Theorem~\ref{theorem:CoplanarLine}.
\begin{proof}[Proof of the main theorem for $l=4$]
Suppose there are atleast two distinct affine hulls. Using Theorem~\ref{theorem:CoplanarLine} let $O$ be the origin and $L=\{t(0,0,0,1)\mid t\in \mbb{R}\}$. 
Let $P_i$ be the plane containing origin spanned by the lines $L$ and $L_i$ for $1\leq i\leq n$. We have by $P_i\neq P_j,1\leq i\neq j\leq n$ because the lines $L_i,1\leq i\leq n$ are mutually skew. Now we project $\mbb{R}^4$ perpendicular to the line $L$ onto $\mbb{R}^3$. Let $\gp:\mbb{R}^4\lra \mbb{R}^3,\gp(x,y,z,t)=(x,y,z)$ be the projection. Then we get a finite configuration of $n$ points $\{\gp(P_i)\mid 1\leq i\leq n\}\subs \mbb{P}^2(\mbb{R})$. The points $\{\gp(P_i)\mid L_i\in \Gl\}$ are collinear in $\mbb{P}^2(\mbb{R})$ where $\Gl$ is the set of lines in a affine hull of a pair of mutually skew lines. Hence the configuration forms an SGC in $\mbb{P}^2(\mbb{R})$. Now by using the basic Sylvester-Gallai theorem we conclude that all points $\gp(P_i),1\leq i\leq n$ are collinear in $\mbb{P}^2(\mbb{R})$. Hence all the lines $L_1,L_2,\cdots,L_n$ lie in one three dimensional space containing origin which is a contradiction to the existence of at least two distinct affine hulls. This completes the proof of the main theorem.	
\end{proof}
Now we prove Theorem~\ref{theorem:CoplanarLine}. 
\begin{proof}
Let $L_1,L_2,\cdots,L_n$ be $n$ mutually skew lines in $\mbb{R}^4$ for some $n\geq 7$. In fact we can even assume that $n\geq 8$ using Corollary~\ref{cor:Fano} because there is at least two distinct affine hulls, all of them containing origin.
So using Lemma~\ref{lemma:AffineHulls} we have that, each line $L_i,1\leq i\leq n$ is in at least three different three dimensional affine hulls. Let us assume that the affine hulls of all these lines are $H_1,H_2,\cdots,H_m$ which are given as follows without loss of generality by gathering as much definite set theoretic knowledge as possible about the configuration.
\equa{H_1 &\os{\sups}{\llra} \{L_1,L_2,L_3,\cdots,L_{a_1}\}=\Gs_1 \\
H_2 &\os{\sups}{\llra} \{L_1,L_{a_1+1},L_{a_1+2},\cdots,L_{a_2}\}=\Gs_2\\
H_3 &\os{\sups}{\llra} \{L_1,L_{a_2+1},L_{a_2+2},\cdots,L_{a_3}\}=\Gs_3\\
\vdots &\llra \hspace{2cm} \vdots\\
	H_{k=b_1}&\os{\sups}{\llra} \{L_1,L_{a_{k-1}+1},L_{a_{k-1}+2},\cdots,L_{a_k=n}\}=\Gs_{b_1}\\
	H_{b_1+1}&\os{\sups}{\llra} \{L_2,L_{a_1+1},\cdots\}=\Gs_{b_1+1}\\
	H_{b_1+2}&\os{\sups}{\llra} \{L_2,L_{a_1+2},\cdots\}=\Gs_{b_1+2}\\	
		\vdots &\llra \hspace{1cm} \vdots\\
H_{b_2}&\os{\sups}{\llra} \{L_2,\cdots\}=\Gs_{b_2}\\
H_{b_2+1}&\os{\sups}{\llra} \{L_3,L_{a_1+1},\cdots\}=\Gs_{b_2+1}\\
H_{b_2+2}&\os{\sups}{\llra} \{L_3,L_{a_1+2},\cdots\}=\Gs_{b_2+2}\\
\vdots &\llra \hspace{1cm} \vdots\\
H_{b_3}&\os{\sups}{\llra} \{L_3,\cdots\}=\Gs_{b_3}\\
\vdots &\llra \hspace{1cm} \vdots\\
H_{b_{a_1-1}}&\os{\sups}{\llra} \{L_{a_1-1},\cdots\}=\Gs_{b_{a_1-1}}\\
H_{b_{a_1-1}+1}&\os{\sups}{\llra} \{L_{a_1},L_{a_1+1},\cdots\}=\Gs_{b_{a_1-1}+1}\\
H_{b_{a_1-1}+2}&\os{\sups}{\llra} \{L_{a_1},L_{a_1+2},\cdots\}=\Gs_{b_{a_1-1}+2}\\
\vdots &\llra \hspace{1cm} \vdots\\
H_{b_{a_1}}&\os{\sups}{\llra} \{L_{a_1},\cdots\}=\Gs_{b_{a_1}}\\
\vdots &\llra \hspace{1cm} \vdots\\
H_m&\os{\sups}{\llra} \{L_{*},\cdots\}=\Gs_{m}.}
Here we observe the following.
\equa{L_1 &\os{\subs}{\llra} \{H_1,H_2,\cdots,H_{b_1}\}=\Gd_1\\
L_2 &\os{\subs}{\llra} \{H_1,H_{b_1+1},\cdots,H_{b_2}\}=\Gd_2\\
L_3 &\os{\subs}{\llra} \{H_1,H_{b_2+1},\cdots,H_{b_3}\}=\Gd_3}
\equa{\vdots &\llra \hspace{2cm} \vdots\\
L_{a_1} &\os{\subs}{\llra} \{H_1,H_{b_{a_1-1}+1},\cdots,H_{b_{a_1}}\}=\Gd_{a_1}\\
L_{a_1+1} &\os{\subs}{\llra} \{H_2,H_{b_1+1},H_{b_2+1},\cdots,H_{b_{a_1-1}+1},\cdots\}=\Gd_{a_1+1}\\
L_{a_1+2} &\os{\subs}{\llra} \{H_2,H_{b_1+2},H_{b_2+2},\cdots,H_{b_{a_1-1}+2},\cdots\}=\Gd_{a_1+2}\\
\vdots &\llra \hspace{2cm} \vdots\\
L_n &\os{\subs}{\llra} \{H_{b_1},\cdots\}=\Gd_n.}
Here we have $b_1=k,a_k=n$ to avoid notation of repeated subscripts. We have $\mid\Gs_i\cap\Gs_j\mid\leq 1,1\leq i\neq j\leq m$ and each $H_i$ is the affine hull of any two lines in $\Gs_i,1\leq i\leq m$. 
We have that all the $H_i,1\leq i\leq m$ are distinct and hence $\dim_{\mbb{R}}(H_i\cap H_j)=2,1\leq i\neq j\leq m$. We also have that $0\in \us{i=1}{\os{m}{\cap}}H_i$.   We also assume that $L_1,L_2,\ldots,L_{n-1}$ does not pass through origin and $L_n$ may or may not pass through origin.
Hence we have $\dim_{\mbb{R}}(\us{j\in \Gd_i}{\cap} H_j)=2,1\leq i\leq n-1$ and $\dim_{\mbb{R}}(\us{j\in \Gd_n}{\cap} H_j)\leq 2$.
\begin{claim}
Let $v_j$ be a normal vector of the hyperplane $H_j,1\leq j\leq m$. The set $\{v_1,v_{b_1},v_{b_2}\}$ is linearly independent and $v_j,1\leq j\leq m$ is linearly dependent on the set $\{v_1,v_{b_1},v_{b_2}\}$.
\end{claim}
\begin{proof}[Proof of Claim]
For $1\leq i\leq n-1$, the set $\{v_j\mid H_j\in \Gd_i\}\subs \mbb{R}^4$ spans a two dimensional space.  Now the set $\{v_1,v_{b_1},v_{b_2}\}$ is linearly independent. Suppose not, then $H_1\cap H_{b_1}\cap H_{b_2}=H_1\cap H_{b_1}=H_1\cap H_{b_2}$. This implies the lines $L_1,L_2$ are coplanar which is a contradiction. Now the vector $v_{b_3}$  is linearly dependent on the set $\{v_1,v_{b_1},v_{b_2}\}$. Suppose not then by applying a suitable linear transformation in $GL_4(\mbb{R})$ we can assume that $v_1=(1,0,0,0),v_{b_1}=(0,1,0,0),v_{b_2}=(0,0,1,0),v_{b_3}=(0,0,0,1)$. Since $n\geq 4$ we have $\dim_{\mbb{R}}(\us{j\in \Gd_i}{\cap} H_j)=2,i=1,2,3$ and hence $v_2=(a,b,0,0),v_{b_1+1}=(c,0,d,0),v_{b_2+1}=(e,0,0,f)$ with $a,b,c,d,e,f\in \mbb{R}\bs\{0\}$. Now the rank of the matrix 
\equ{\matthreefour{a}{b}{0}{0}{c}{0}{d}{0}{e}{0}{0}{f}} is three. Hence the space $H_2\cap H_{b_1+1}\cap H_{b_2+1}$ is exactly one dimensional and contains the line $L_{a_1+1}$. But $n>a_1+1$ and the line $L_{a_1+1}$ does not pass through origin which is a contradiction. So $v_{b_3}$ is linearly dependent on the set $\{v_1,v_{b_1},v_{b_2}\}$. Similarly by applying the same argument, we conclude that
$v_{b_j}$ is linearly dependent on the set $\{v_1,v_{b_1},v_{b_2}\}$ for $3\leq j\leq a_1$ since $H_2\cap H_{b_1+1}\cap H_{b_{j-1}+1}$ contains the line $L_{a_1+1}$. Now $v_i$ is linearly dependent on $\{v_1,v_{b_j}\}$ for $b_{j-1}+1\leq i\leq b_j-1,3\leq j\leq a_1$ and hence we obtain that $v_i$ is linearly dependent on the set $\{v_1,v_{b_1},v_{b_2}\}$ for $1\leq i\leq b_{a_1}$. Now we observe the following. 
For any $2\leq i\leq n$ there exists a unique $j\in \{1,2,\cdots,b_1\}$ such that $L_i\subs H_j$. Similarly for $1\leq i\leq n,i\neq 2$ there exists a unique $j\in \{1,b_1+1,b_1+2,\cdots, b_2\}$ such that $L_i\subs H_j$.  Now we observe that the set $\Gd_n\subs \us{i=1}{\os{n-1}{\cup}} \Gd_i$. So by applying the dual argument we get that, for $a_1+1\leq i\leq n-1$, the set $\Gd_i$ contains two hyperplanes one from the set $\Gd_1$ and one from the set $\Gd_2$ different from $H_1$. The normal vector of any hyperplane in $\Gd_i$ is linearly dependent on the normal vectors of those two hyperplanes in $\Gd_i$ coming from $\Gd_1$ and $\Gd_2$. So for $b_{a_1}+1\leq i\leq m, v_i$ is linearly dependent on the set $\{v_1,v_{b_1},v_{b_2}\}$. This can be concluded irrespective of the apriori fact that $\dim_{\mbb{R}}(\us{j\in \Gd_n}{\cap} H_j)=0$ or $1$ or $2$.  This completes the proof of the claim.
\end{proof}

Without loss of generality by applying a linear transformation in $GL_4(\mbb{R})$, let  $v_1=(1,0,0,0),v_{b_1}=(0,1,0,0),v_{b_2}=(0,0,1,0)$. Then we get using the claim that, the line $\{t(0,0,0,1)\mid t\in \mbb{R}\}=L\subs \us{i=1}{\os{m}{\cap}} H_i$. Now for $1\leq i\leq n$, the line $L_i$ and the line $L$ are contained in a two dimensional plane $H_r\cap H_s$ for any $r\neq s$ such that $H_r,H_s\in \Gd_i$. So in that plane the line $L_i$ and $L$ are coplanar. Now $L\neq L_i,1\leq i\leq n-1$ because $L$ passes through origin and $L_i$ does not. We also have that $L\neq L_n$.

This completes the proof of Theorem~\ref{theorem:CoplanarLine}.
\end{proof}
\section{\bf{Proof of the Main Theorem in General}}
Now state the analogous statement of Theorem~\ref{theorem:CoplanarLine} in higher dimensions. 
\begin{theorem}
	\label{theorem:CoplanarLineinHigherDimensions}
	Let $l\geq 4$ be an integer.
	Suppose there is a finite family $\Gs=\{L_1,L_2,\cdots,$ $L_n\}$ of mutually skew lines in $\mbb{R}^l$ such that the three dimensional affine span (hull) of every two lines in $\Gs$, contains at least one more line of $\Gs$. Let $\{H_1,H_2,\cdots,H_m\}$ be the distinct three dimensional affine hulls and $m\geq 2$. If the hyperplane arrangement of affine hulls is central, that is, $O\in \us{i=1}{\os{m}{\cap}} H_i$, then there exists a line $L\subs \mbb{R}^l$ different from  $L_i$ and passing through $O$ such that the line $L_i$ and $L$ are coplanar for each $1\leq i\leq n$.
\end{theorem}
We first prove the main theorem using this Theorem~\ref{theorem:CoplanarLineinHigherDimensions}.
\begin{proof}[Proof of the main theorem]
	This proof is similar to the proof of the main theorem for the value $l=4$ given in the previous section.
\end{proof}
Now we prove Theorem~\ref{theorem:CoplanarLineinHigherDimensions}.
\begin{proof}
Just similar to the proof of Theorem~\ref{theorem:CoplanarLine}, here also we have a similar set theoretic knowledge of the lines and their affine hulls. We assume as usual that $O$ is the origin and $L_i,1\leq i\leq n-1$ do not pass through origin and $L_n$ may or may not pass through origin.
Here we have that all the affine hulls $H_i,1\leq i\leq m$ are distinct and $\dim_{\mbb{R}}(H_i\cap H_j)=2$ for all $i\neq j$ such that $H_i,H_j\in \Gd_p$ for any $1\leq p\leq n-1$ and in general we have $\dim_{\mbb{R}}(H_i\cap H_j)\leq 2$ unlike the case when $l=4$. It follows that $\dim_{\mbb{R}}(\us{j\in \Gd_i}{\cap} H_j)=2,1\leq i\leq n-1$ and $\dim_{\mbb{R}}(\us{j\in \Gd_n}{\cap} H_j)\leq 2$.  Now we prove the following claim.
\begin{claim}
Let $V_i=H_i^{\perp}\subs \mbb{R}^l,1\leq i\leq m$. Then 
\begin{enumerate}
	\item $\dim_{\mbb{R}}(V_1+V_{b_1}+V_{b_2})=l-1$.
	\item $V_i\subsetneq V_1+V_{b_1}+V_{b_2}$ for $1\leq i\leq m$.
\end{enumerate}
\end{claim}
\begin{proof}[Proof of Claim]
We prove $(1)$ first. The space $(V_1+V_{b_1}+V_{b_2})=(H_1\cap H_{b_1}\cap H_{b_2})^{\perp}$.
We have \equa{\dim_{\mbb{R}}(H_1\cap H_{b_1}\cap H_{b_2})&=\dim_{\mbb{R}}(H_1\cap H_{b_1})+\dim_{\mbb{R}}(H_1\cap H_{b_2})\\&-\dim_{\mbb{R}}((H_1\cap H_{b_1})+(H_1\cap H_{b_2})).}
Now for $i=1,2, H_1\cap H_{b_i}$ is the plane containing $L_i$ and the origin $O$.	Hence the space $(H_1\cap H_{b_1})+(H_1\cap H_{b_2})$ is exactly the affine hull $H_1$ of the lines $L_1$ and $L_2$ containing $O$ which is therefore three dimensional.
So we have $\dim_{\mbb{R}}(H_1\cap H_{b_1}\cap H_{b_2})=1$ and hence $\dim_{\mbb{R}}(V_1+V_{b_1}+V_{b_2})=l-1$. Now we prove $(2)$. First we observe that for $1\leq p \leq n-1, i\neq j,r\neq s,i,j,r,s\in \Gd_p$ the $V_i+V_j=V_r+V_s,\dim_{\mbb{R}}(V_i+V_j)=l-2$. This is because $H_i\cap H_j=H_r\cap H_s=\us{t\in \Gd_p}{\cap} H_t$ and its dimension is $2$. Now we have the following sequence of subspace inclusions. \equa{V_{b_3} &\subsetneq V_1+V_{b_2+1}\\ 
	&\subs V_1+V_2+V_{b_1+1} (\text{because }V_{b_2+1}\subsetneq V_2+V_{b_1+1})\\
 &=V_1+V_2+V_{b_2}(\text{because }V_1+V_{b_1+1}= V_1+V_{b_2})\\
&=V_1+V_{b_1}+V_{b_2}(\text{because }V_1+V_2= V_1+V_{b_1}).}
Similarly we get $V_{b_j}\subsetneq V_1+V_{b_1}+V_{b_2}$ for $3\leq j\leq a_1$. This implies that $V_i\subsetneq V_1+V_{b_1}+V_{b_2}$ for $1\leq i\leq b_{a_1}$. We have $\Gd_n\subs \us{i=1}{\os{n-1}{\cup}}\Gd_i$ and for $a_1+1\leq i\leq n-1$, the set $\Gd_i$ contains two affine hulls one from the set $\Gd_1$ and one from the set $\Gd_2$ different from $H_1$. Hence $V_i\subsetneq V_1+V_{b_1}+V_{b_2},1\leq i\leq m$. This proves the claim.
\end{proof}
Coming to the proof of Theorem~\ref{theorem:CoplanarLineinHigherDimensions}, let $L=(V_1+V_{b_1}+V_{b_2})^{\perp}\subseteq \us{i=1}{\os{m}{\cap}} H_i$. For $1\leq i\leq n$, the line $L$ and $L_i$ are coplanar in the plane $H_r\cap H_s$ for $r\neq s,r,s\in \Gd_i$. Clearly $L\neq L_i,1\leq i\leq n-1$ and by coplanarity of $L$ with say $L_1$, we must have $L\neq L_n$. This proves Theorem~\ref{theorem:CoplanarLineinHigherDimensions}.
\end{proof}
\section{\bf{On (2,5)-Representations}}
We mention a definition first.
\begin{defn}
Let $V$ be a finite dimensional vector space over a field $\mbb{K}$. We say two affine spaces $W_1,W_2\subs V$ are skew if \equ{\dim_{\mbb{K}}W_1+\dim_{\mbb{K}}W_1+1=\dim_{\mbb{K}}(\text{AH}(W_1,W_2))}
where $\text{AH}(W_1,W_2)=\{tw_1+(1-t)w_2\mid w_i\in W_i,i=1,2\}$ is the affine hull of $W_1$ and $W_2$.
\end{defn}
The analogous result of main Theorem~\ref{theorem:Main} for $(2,5)$-representations is not true as we mention in the following remark. Example~\ref{example:25Representation} in the support of this remark is given later in the section.
\begin{remark}
\label{remark:HigherDim}
Let $\Gs$ be a finite family of mutually skew affine spaces in $\mbb{R}^6$ of dimension $2$. Suppose the $5$-dimensional affine span (hull) of every two affine spaces in $\Gs$ is a linear subspace containing origin and contains at least one more affine space of $\Gs$. Then $\Gs$ need not be entirely contained in a single $5$-dimensional linear subspace.
\end{remark}

There have been some results in the direction of this remark. Instead of $\Gs\subs\mbb{R}^6$, we suppose $\Gs\subs \mbb{R}^l$ with $l\geq 11$. Using the bound given in Theorem $1.6$ in~\cite{MR3775994}, it can be shown that the set $\Gs$ is contained in a $10$-dimensional space. This can be obtained as follows. Assume that $\Gs\subs \mbb{R}^l=\{z_{l+1}=1\} \subs \mbb{R}^{l+1}$. Consider the cones of affine spaces in $\Gs$ over the origin in $\mbb{R}^{l+1}$. They are $3$-dimensional linear subspaces of $\mbb{R}^{l+1}$, which intersect pairwise at the origin. Now, we complexify these linear subspaces to get $3$-dimensional linear subspaces in $\mbb{C}^{l+1}$ intersecting pairwise at origin. Here we apply Theorem $1.6$ in~\cite{MR3775994}, to obtain the dimension bound of the span of all $3$-dimensional spaces obtained from $\Gs$ to be $11$. This bound will reduce to $10$ after intersecting with the hyperplane $z_{l+1}=1$. This bound holds even if we assume that the $5$-dimensional affine hulls of any two skew affine spaces in $\Gs$ need not contain origin in $\mbb{R}^l$.

Now we mention another related remark. 
\begin{remark}
	\label{remark:HigherDimTwo}
Suppose we have a finite family $\Gs$ linear $2$-dimensional subspaces of $\mbb{R}^5$ such that any two linear subspaces in $\Gs$ intersect only at origin. Also suppose that in the $4$-dimensional subspaces spanned by two linear subspaces of $\Gs$ there exists at least one more linear subspace from $\Gs$. Then $\Gs$ need not be contained in a single $4$-dimensional subspace. The example just below supports this remark.
\end{remark}
\begin{example}
Consider the more generic enough hyperplane arrangement of Example~\ref{example:FanoPlane} in $\mbb{R}^4$ where the seven mutually skew lines corresponding to the fano plane do not lie in a single three dimensional subspace. Place this arrangement in the subset $\mbb{R}^4=\{z_5=1\}\subs \mbb{R}^5$. Now consider the cones of these seven lines in $\{z_5=1\}$ over origin in $\mbb{R}^5$. Then we get a family $\widetilde{\Gs}=\{P_1,P_2,\cdots,P_7\}$ of seven planes in $\mbb{R}^5$ which pairwise intersect at the origin. So we obtain that  all the seven planes cannot lie in a single four dimensional spaces. Otherwise upon intersecting with $\{z_5=1\}$, this implies that all the seven lines in a single three dimensional space which is a contradiction.
\end{example}

Now we mention a possibility about what could happen to Remark~\ref{remark:HigherDim}.

\begin{example}
\label{example:25Representation}
Let $\widetilde{\Gs}=\{P_1,P_2,\cdots,P_7\}$ be the example obtained in the support of Remark~\ref{remark:HigherDimTwo}. Here $n=7$. Now let $L=\{t(0,0,0,0,0,1)\mid t\in \mbb{R}\}$. Lift the arrangement under the inverse image of the projection $\gp:\mbb{R}^6\lra \mbb{R}^5,\gp(x,y,z,t,u,v)=(x,y,z,t,u)$. So we have $C_i=\gp^{-1}(P_i),1\leq i\leq 7$ the three dimensional spaces such that $C_i\cap C_j=L,1\leq i\neq j\leq 7$ which are obtained from the fano plane $\mbb{P}^2(\mbb{F}_2)$ whose $(1,3)$-representation comes from a more generic enough hyperplane arrangement  in $\mbb{R}^4$. We can find two dimensional affine subspaces $A_i\subs C_i\subs \mbb{R}^6, 1\leq i\leq 7$ such that $A_i\cap L=\{(0,0,0,0,0,t_i)\},t_i\neq 0,t_i\neq t_j,1\leq i\neq j\leq 7$ which are mutually skew and $C_i$ is the cone of $A_i$ in $\mbb{R}^6$. Then the affine hulls of $A_i$ and $A_j$ for $1\leq i\neq j\leq 7$ is exactly $C_i+C_j$ which passes through origin and is $5$-dimensional. The affine hulls $C_i+C_j,1\leq i\neq j\leq 7$ do not lie in a single five dimensional space in $\mbb{R}^6$. 
\end{example}


\begin{thebibliography}{1}

\bibitem{MR1142273}
J.~Bokowksi, J.~R.~Gerbert,
\newblock {A} new {S}ylvester-{G}allai configuration representing the 13-point projective plane in $\mbb{R}^4$,
\newblock {\em Journal of Combinatorial Theory, Series B}, Vol. {\bf 54}, Issue {\bf 1}, Jan. 1992, pp. 161-165, \url{https://doi.org/10.1016/0095-8956(92)90075-9}, MR1142273

\bibitem{MR0996767}
E.~Boros, Z.~F\"{u}redi, L.~M.~Kelly,
\newblock {O}n representing sylvester-gallai designs,
\newblock {\em Discrete and Computational Geometry}, {\bf 4}, no. {\bf 4}, 1989, pp. 345-348, \url{https://doi.org/10.1007/BF02187735}, MR0996767

\bibitem{MR1069788}
P.~Borwein, W.~O.~J.~Moser,
\newblock {A} survey of {S}ylvester's {P}roblem and its generalizations,
\newblock {\em Aequationes Mathematicae}, Vol. {\bf 40}, 1990, pp. 111-135, \url{https://doi.org/10.1007/BF02112289}, MR1069788

\bibitem{MR2163782}
P.~Brass, W.~O.~J.~Moser, P.~Janos, 
\newblock {R}esearch Problems in discrete geometry,
\newblock {Springer New York}, 2005, pp. xii+499, ISBN-13: 978-0387-23815-8, ISBN-10:  0-387-23815-8, \url{https://doi.org/10.1007/0-387-29929-7}, MR2163782

\bibitem{MR3775994}
Z.~Dvir, G.~Ankit, R.~Oliveira, J.~Solymosi,
\newblock {R}ank bounds for design matrices with block entries and geometric applications, 
\newblock {\em Discrete Analysis}, no. {\bf 5}, 24 pages, \url{http://dx.doi.org/10.19086/da.3118} MR3775994

\bibitem{MR2202107}
N.~Elkies, L.~M.~Pretorius, K.~J.~Swanpoel,
\newblock {S}ylvester-{G}allai {T}heorems for {C}omplex {N}umbers and {Q}uaternions,
\newblock {\em Discrete and Computational Geometry}, {\bf 35}, no. {\bf 3}, 2006, pp. 361-373, \url{https://doi.org/10.1007/s00454-005-1226-7}, MR2202107

\bibitem{MR0698647}
P.~Erd\H{o}s,
\newblock {P}ersonal reminiscences and remarks on the mathematical work of {T}ibor {G}allai, 
\newblock {\em Combinatorica}, Vol. {\bf 2}(3), 1982, pp. 207-212, \url{https://doi.org/10.1007/BF02579228}, MR0698647 

\bibitem{MR1525919}
P.~Erd\H{o}s, R.~Steinberg,
\newblock {P}roblems and {S}olutions: {A}dvanced {P}roblems: {S}olutions: 4065,
\newblock {\em American Mathematical Monthly}, Vol. {\bf 51}, no. {\bf 3}, Mar. 1944, pp. 169-171, \url{ https://www.jstor.org/stable/2303021}, MR1525919

\bibitem{MR0717609}
F.~Hirzebruch, 
\newblock {A}rrangements of lines and algebraic surfaces,
\newblock {\em In: Artin M., Tate J. (eds) Arithmetic and Geometry, Papers Dedicated to I.R. Shafarevich on the Occasion of His Sixtieth Birthday}, Vol. {\bf II}, pp. 113-140, Progress in Mathematics, Vol. {\bf 36}, Birkh\"{a}user, Boston, MA 1983, \url{https://doi.org/10.1007/978-1-4757-9286-7_7}, MR0717609

\bibitem{MR0834051}
L.~M.~Kelly,
\newblock {A} resolution of the sylvester-gallai problem of J.~P.~Serre,
\newblock {\em Discrete and Computational Geometry} {\bf 1}, no. {\bf 2}, 1986, pp. 101-104, \url{https://doi.org/10.1007/BF02187687}, MR0834051

\bibitem{MR0460343}
Y.~Miyaoka,
\newblock {O}n the {C}hern numbers of surfaces of general type,
\newblock {\em Inventionae Mathematicae}, Vol. {\bf 42}, 1977, pp. 225-237, \url{https://doi.org/10.1007/BF01389789}, MR0460343 

\bibitem{JPSERRE}
J.~P.~Serre,
\newblock {A}dvanced {P}roblem 5359,
\newblock {\em American Mathematical Monthly}, Vol. {\bf 73}, no. {\bf 1}, Jan. 1966, pp. 87-89, \url{https://www.jstor.org/stable/2313941}

\bibitem{JJS}
J.~J.~Sylvester,
\newblock {M}athematical {Q}uestion 11851,
\newblock {\em Educational Times}, Vol. {\bf 59}, 1893, pp. 98.

\bibitem{MR0451180}
S.~T.~Yau,
\newblock {C}alabi's conjecture and some new results in algebraic geometry,
\newblock {\em Proceedings of the National Academy of Sciences of the USA}, Vol. {\bf 74}, no. {\bf 5}, 1977, pp. 1798-1799, \url{https://doi.org/10.1073/pnas.74.5.1798}, MR0451180
\end{thebibliography}
\end{document}